\documentclass[12pt]{amsart}
\DeclareUnicodeCharacter{1D43B}{H}
\DeclareUnicodeCharacter{1D43F}{L}
\DeclareUnicodeCharacter{1D44A}{W}

\usepackage{graphicx}
\usepackage{subfig}
\usepackage{tabularx}

\usepackage{mathrsfs}     
\usepackage{helvet}         
\usepackage{courier}        
\usepackage{type1cm}      
\usepackage{color}
\usepackage{todonotes}

\usepackage{geometry,calc,color}
\usepackage{amsmath,amssymb}

\usepackage{enumerate}
\usepackage{esint}

\newtheorem{theorem}{Theorem}[section]
\newtheorem{corollary}[theorem]{Corollary}
\newtheorem{lemma}[theorem]{Lemma}
\newtheorem{proposition}[theorem]{Proposition}
\newtheorem{comment}[theorem]{Comment}
\theoremstyle{remark}
\newtheorem{remark}[theorem]{Remark}

\newcommand{\diam}[1]{\mathrm{diam}(#1)}
\numberwithin{equation}{section}

\long\def\NOTE#1{} 

\newcommand{\G}{\Gamma}
\newcommand{\y}{{\mathbf{y}}}
\newcommand{\x}{{\mathbf{x}}}
\newcommand{\patch}[1]{{\tau_{#1}}}
\newcommand{\Nodes}{\mathcal{N}}

\renewcommand{\b}[1]{\phi_{#1}}

\newcommand{\Th}{\mathcal{T}_h}
\newcommand{\phistar}[1]{\phi^*_{#1}}

\newcommand{\supp}{\mathop{\mathrm{supp}}}
\newcommand{\dist}{\mathrm{d}}

\newcommand{\tH}{\widetilde H}

\renewcommand{\phi}{\varphi}

\renewcommand{\O}{\Omega}
\newcommand{\Ext}{\mathscr{E}}

\newcommand{\Span}{\mathrm{span}\,}

\newcommand{\hB}{\widehat B}
\newcommand{\htau}{\widehat\tau}
\newcommand{\hy}{\widehat y}
\newcommand{\RR}{\mathbb{R}}


\begin{document}

\title[Localization of trace norms]{Localization of trace norms in two and three dimensions}%

\author[S. Bertoluzza]{Silvia Bertoluzza}
\address{IMATI ``E. Magenes'', CNR, Pavia (Italy)}%
\email{silvia.bertoluzza@imati.cnr.it}%

%

\date{\today}
\thanks{This work was cofunded by MIUR Progetti di Ricerca di Rilevante Interesse Nazionale (PRIN) Bando 2020 (grant 20204LN5N5). The author is a member of the INdAM Research group GNCS}%

\begin{abstract} 
		We extend a localization result for the $H^{1/2}$ norm by B. Faermann \cite{faermann2000localization,faermann2002localization}, to a wider class of  subspaces of $H^{1/2}(\G)$, and we prove an analogous result for the  $H^{-1/2}(\G)$ norm, $\G$ being the boundary of a bounded polytopal domain $\Omega$ in $\mathbb{R}^n$, $n=2,3$. As a corollary, we obtain equivalent, better localized, norms for both $H^{1/2}(\G)$ and $H^{-1/2}(\G)$, which can be exploited, for instance, in the design of preconditioners or of stabilized methods.
\end{abstract}
\maketitle


\date{\today}

\maketitle
\section{Introduction}
The spaces $H^{1/2}(\G)$ and $H^{-1/2}(\G)$, $\G$ being the boundary of a domain $\O$ in $\mathbb{R}^n$, naturally appear in the analysis and in the discretization of second order elliptic boundary value problems. Indeed, they  
are the natural spaces for the trace of, respectively,  the $H^1(\O)$ solution and  its normal flux. The capability of numerically evaluating such norms in an efficient way would be beneficial for several aspects of the numerical solution of such a class of problems, from preconditioning, to stabilization, to a posteriori error estimation in adaptivity. Unfortunately, the scalar product for such spaces  is of an intrinsecally global nature: the supports of two elements being well separated does not imply that their scalar product is zero or even only small. This makes computations involving the evaluation of such  norms and/or of the corresponding scalar products computationally heavy, and often inpractical. In particular situations, it is however possible to localize the $H^{1/2}(\G)$ norm:  the works of Faermann \cite{faermann2000localization,faermann2002localization} 
show that, if we restrict ourselves to the subspace of functions that are orthogonal, in $L^2(\G)$, to suitable finite element spaces, the $H^{1/2}(\G )$ norm is equivalent to a norm with good localization properties,  obtained as the squared root of the sum of squared local norms on overlapping patches. The aim of this paper is, on the one hand, to extend such  results to a larger set of subspaces of $H^{1/2}(\G)$ and, on the other hand, to prove an analogous result for the norm of the dual space $H^{-1/2}(\G)$.

\section{Functional framework}

Let $\O$ denote a bounded, simply connected, possibly curvilinear polygon  in $\mathbb{R}^n$, $n=2,3$, with Lipschitz boundary $\G = \partial\Omega$. We assume that $\O$ in non degenerate, in the sense that  the geodesic distance on $\G$ is bounded from above by a constant times the euclidean distance in $\mathbb{R}^n$.

 We denote by $\langle \cdot, \cdot \rangle_\tau$, $\tau \subseteq \G$ being either $\Gamma$ or a simply connected subdomain of $\Gamma$, the $L^2(\tau)$ scalar product:
\[
\langle v, w \rangle_{\tau} = \int_\tau v(s) w (s)\,ds.
\]
 We recall that the norm for the space $H^{1/2}(\tau)$  can be defined as
\[
\| v \|^2_{1/2,\tau} = \| v \|_{0,\tau}^2 + | v |_{1/2,\tau}^2\quad \text{ with } \quad | v |_{1/2,\tau}^2 = \int_\tau \,dt \int_\tau \,ds \frac{| v(t) - v(s) |^2}{| t - s |^n}.
\]
We also recall that the space $H^{1/2}_{00}(\tau)$ can be defined as the subspace of those functions $v$ in  $H^{1/2}(\tau)$ such that $\Ext(v)$ belongs to $H^{1/2}(\Gamma)$, $\Ext(v)\in L^2(\G)$ being the function coinciding with  $v$ in $\tau$ and identically vanishing in $\Gamma\setminus\tau$. The $H^{1/2}_{00}(\tau)$ norm is defined as\footnote{Here and thorughout the paper, $A \lesssim B$ (resp. $A\gtrsim B$) stands for $A \leq C B$, for some positive constant $C$, which, when a mesh $\Th$ is involved in the definition of the quantities $A$ and $B$, is assumed to be independent of its meshsize $h$. The expression $A \simeq B$ stands for $A \lesssim B \lesssim A$.
}
\begin{equation}\label{onehalfzznorm}
	\| v \|^2_{H^{1/2}_{00}(\tau)} = | \Ext(v) |^2_{1/2,\G} \simeq | v |^2_{1/2,\tau} + \diam{\tau}^{n-2}\int_\tau \frac{ | v(s) |^2}{\dist(s,\partial\tau)^{n-1}}\,ds.
\end{equation}
By abuse of notation, here and in the following, for all $v \in H^{1/2}_{00}(\tau)$ we will denote its extension by the same letter $v$ instead of using the heavier notation $\Ext(v)$.

We let $H^{-1/2}(\tau)=(H^{1/2}(\tau))'$ and, for $\tau$ simply connected subdomain of $\Gamma$, $\tH^{-1/2}(\tau) = (H^{1/2}_{00}(\tau))'$, denote the dual spaces of $H^{1/2}(\tau)$ and $H^{1/2}_{00}(\tau)$ respectively. Letting, by abuse of notation $\langle \cdot,\cdot \rangle_\tau$ denote both the duality pairing of $H^{-1/2}(\tau)$ and $H^{1/2}(\tau)$ and of $\tH^{-1/2}(\tau)$ and $H^{1/2}_{00}(\tau)$, the dual spaces are respectively endowed   with the norms
\[
\| \phi \|_{-1/2,\tau} = \sup_{{v\in H^{1/2}(\tau)}\atop {v \not=0}} \frac{\langle \phi, v \rangle_\tau}{\| v \|_{1/2,\tau}}, \qquad \| \phi \|_{\tH^{-1/2}(\tau)} = \sup_{{\phi\in H^{1/2}_{00}(\tau)}\atop{\phi\not=0} } \frac{\langle \phi, v \rangle_\tau}{\| v \|_{H^{1/2}_{00}(\tau)}}.
\]

In the following, we will make use of the injection bound of $L^2(\tau)$ in $L^1(\tau)$, that reads
\begin{equation}\label{2.2a}
\| v \|_{L^1(\tau)} \leq | \tau |^{1/2} \| v \|_{0,\tau}, \qquad \text{ for all } v \in L^2(\tau),
\end{equation}
$| \tau|$  denoting the measure of $\tau$.
Moreover, we will need the following Poincar\'e type inequalities:
\begin{gather}\label{poincare}
\| v \|_{0,\tau} \lesssim \diam\tau^{1/2} \| v \|_{H^{1/2}_{00}(\tau)}, \qquad \text{ for all } v \in H^{1/2}_{00}(\tau),\\
\| v \|_{0,\tau} \lesssim \diam\tau^{1/2} | v |_{1/2,\tau}, \qquad \text{ for all } v \in H^{1/2}(\tau) \quad \text{ with }\int_\tau v(s)\,ds = 0,\label{poincareavfree}
\end{gather}
the implicit constant in \eqref{poincare} being independent of the shape of $\tau$, the one  in \eqref{poincareavfree} depending on the ratio between the diameter of $\tau$ and the diamater of the largest ball inscribed in $\tau$.

\NOTE{Per far tornare tutto, bisogna che la parte di bordo della norma $H^{1/2}_{00}
$ sia scalata con $h^{n-2}$. Potrebbe essere giusto, verificare.
\[
\int_\Gamma \,dx \int_\Gamma \, dy \cdots =
\int_\tau \int_\tau \cdots + 2 \int_\tau \, dx \int_{\G\setminus\tau}\,dy \frac {| v(x) |}
{|x - y|^n} = \int_{\tau} \,dx | v(x) |^2 \int_{\G \setminus \tau} \,dy \frac1{| x - y |^n}
\]
How do I bound 
\[
\int_{\G \setminus \tau} \,dy \frac1{| x - y |^n}
\]

I take two balls centered in $x$ such that 
\[
B_0(x) \subseteq \tau \subseteq B_1(x)
\]
I have that the radius of $B_0(x)$ is $\sim \dist(x,\partial\tau)$ and the radius of $B_1$ is $\sim h = \text{diam}(\tau)$. I do a change of variable bringing $B_1$ in $\hB$: 
$x = 0$ (so that is unchanged), $\hy = y/h$, $d\hy = dy h^{1-n}$, $\dist(0,\partial\htau) = \dist(0,\partial\tau)/h$. 
I have
\[
\int_{\RR^n \setminus \tau} \frac 1 {| y |^n} \,dy =
	\int_{\RR^n \setminus \htau} \frac 1 {| h \hy |^n} h^{n-1 }d \hy = h^{-1} \int_{\RR^n \setminus \htau} \frac 1 {|\hy|^n} d \hy \sim h^{-1} \frac 1 {\dist(0,\partial \htau)^{n-1}} \sim \frac {h^{n-2}} {\dist{(0,\partial \tau)}}
\]
}

\NOTE{
The best constants in the Poincar\'e inequality may depend on the shape of the domain, but they are uniformly bounded by a universal constant -- Need to check this.
}

\section{Localization of the $H^{1/2}(\G)$ norm}

 Let $\Th$ denote a quasi uniform, shape regular decomposition of $\G = \partial \O$ into (possibly curvilinear) segments/triangle, of mesh size $h$,  $\Nodes$ denoting the corresponding set of nodes. For each node $\y \in \Nodes$ we let \[
\patch{\y} = \cup_{T\in \Th:\ \y \in \partial T}  T \] denote the patch of elements sharing $\y$ as a vertex. 
Let now $v \in H^{1/2}(\G)$, and assume that $v = \sum_{\y\in \Nodes} v^\y$, with $v^\y \in  H^{1/2}_{00}(\patch{\y})$ for all $\y$.
By the definition of the $H^{1/2}_{00}(\patch{\y})$ norm, we immediately see that 
\begin{equation}\label{nonoptimal}
| v |_{1/2,\G} = | 	\sum_{\y \in \Nodes}  v^\y |_{1/2,\G}  \leq 
\sum_{\y \in \Nodes} | v^\y |_{1/2,\G} =
\sum_{\y\in \Nodes} \| v^{\y} \|_{H^{1/2}_{00}(\patch{\y})}.
\end{equation}
However, a stronger bound holds, as stated by the following lemma. 
\begin{lemma}\label{lem:locposnew}
Assume that $v \in H^{1/2}(\G)$ satisfies
	\begin{equation}\label{decomposition}
	v =  \sum_{\y\in \Nodes} v^\y, \qquad \text{ with,\quad for all }\ \y\in \Nodes, \quad v^{\y} \in H^{1/2}_{00} (\patch{\y})
	.\end{equation}
	Then we have
	\begin{equation}\label{boundonehalfnew}
	\vert v \vert_{1/2,\G} \lesssim \sqrt{\sum_{\y\in \Nodes} \| v^{\y} \|^2_{H^{1/2}_{00}(\patch{\y})}}.
	\end{equation}
\end{lemma}


\begin{proof} Let $v \in H^{1/2}(\G)$ satisfy \eqref{decomposition}. We have
	\[
	| v |^2_{1/2,\G}  = (v,v)_{1/2} = \sum_{\y\in\Nodes} \sum_{\x\in\Nodes}(v^{\y},v^{\x})_{1/2}, 
	\]
	where 
	\[(v,w)_{1/2}
	=
	\int_{\G } \,dt \int_{\G} \,ds \frac{(v(s) -v(s))(w(s) - w(t))}{| s - t |^{n} }.
	\]
	Let us bound $(v^\x,v^\y)_{1/2}$ for $\x, \y$ such that $\dist(\patch{\x},\patch{\y}) > 0$. As in such a case, for $s \in  \patch{\y}$ and $t \in \patch{\x}$, we have that $| t - s | \geq \dist(\patch{\x},\patch{\y})$, we can write
	\begin{multline}
		(v^\x,v^\y)_{1/2}
		=
		\int_{\G } \,dt \int_{\G} \,ds \frac{(v^\x(s) -v^\x(t))(v^\y(s) - v^\y(t))}{| s - t |^n} = \\[2mm]
		\int_{\G } \,dt \int_{\G} \,ds  \frac{v^\x(s) v^\y(s)}{| s - t |^n} + \int_{\G } \,dt \int_{\G} \,ds  \frac{v^\x(t) v^\y(t)}{| s - t |^n} - 2 \int_{\G } \,dt \int_{\G} \,ds  \frac{v^\x(t) v^\y(s)}{| s - t |^n}  \\[2mm] =
		- 2 \int_{\G } \,dt \int_{\G} \,ds  \frac{v^\x(t) v^\y(s)}{| s - t |^n} 
		\lesssim 2 \int_{\G } \,dt \int_{\G} \,ds  \frac{| v^\x(t)|\, | v^\y(s)| }{| s - t |^n}\\[2mm] \lesssim 
		\frac{2}{\dist(\patch{\x},\patch{\y})^n}
		\left(\int_{\patch{\x}} | v^\x(t) | \,dt \right) \, \left( \int_{\patch{\y}} | v^\y(s) | \,ds \right) =
		 \frac{2}{\dist(\patch{\x},\patch{\y})^n} \| v^\x \|_{L^1(\patch{\x})}  \| v^\y \|_{L^1(\patch{\y})}
\\[2mm]		\lesssim \frac{2h^{n-1}}{\dist(\patch{\x},\patch{\y})^n}  \| v^\x \|_{0,\patch{\x}} \| v^\y \|_{0,\patch{\y}} 
		\lesssim \frac{2h^n}{\dist(\patch{\x},\patch{\y})^n} \| v^\x \|_{H^{1/2}_{00}(\patch{\x})} \| v^\y \|_{H^{1/2}_{00}(\patch{\y})},
	\end{multline}
	where  we used \eqref{2.2a} to bound the $L^1$ norms with the $L^2$ norms, which we then bound using  \eqref{poincare}. For $\x, \y$ such that $\dist(\patch{\x}, \patch{\y}) = 0$,  we have instead the following bound: 
	\begin{multline*}
		(v^\x,v^\y)_{1/2} = \int_{\patch{\x} \cup \patch{\y} }\,dt \int_{\patch{\x} \cup \patch{\y}} \,ds \frac{(v^\x(s) -v^\x(t))(v^\y(s) - v^\y(t))}{| s - t |^2} \\
		\leq | v^\x |_{1/2.\patch{\x} \cup \patch{\y} } | v^\y |_{1/2,\patch{\x} \cup \patch{\y} } \lesssim \| v^\x \|_{H^{1/2}_{00}(\patch{\x})} \| v^\y \|_{H^{1/2}_{00}(\patch{\y})} \leq  \| v^\x \|^2_{H^{1/2}_{00}(\patch{\x})} + \| v^\y \|^2_{H^{1/2}_{00}(\patch{\y})}.
	\end{multline*}
	 Then we have
	\begin{multline} \label{bound1}
		| v |_{1/2}^2 \leq \sum_{\y\in \Nodes} \sum_{{\x \in \Nodes} \atop{\dist(\patch{\x},\patch{\y}) = 0}} (v^\x ,v^\y)_{1/2} + \sum_{\y\in \Nodes} \sum_{{\x \in \Nodes} \atop{\dist(\patch{\x},\patch{\y}) > 0}} (v^\x ,v^\y)_{1/2} \\[2mm]
		\lesssim
		\sum_{\y\in \Nodes} \sum_{{\x \in \Nodes} \atop{\dist(\patch{\x},\patch{\y}) = 0}} \Big(
		\| v^\x \|^2_{H^{1/2}_{00}(\patch{\x})} + \| v^\y \|^2_{H^{1/2}_{00}(\patch{\y})}
		\Big)
		 +  \sum_{\y\in \Nodes} \sum_{{\x \in \Nodes} \atop{\dist(\patch{\x},\patch{\y}) > 0}} \frac{2h^n}{\dist(\patch{\x},\patch{\y})^n} \| v^\x \|^2_{H^{1/2}_{00}(\patch{\x})},
	\end{multline}
	where we used that, by the shape regularity of the mesh the cardinality of the set $\{\x: \dist(\patch\x,\patch\y) = 0\}$ is uniformly bounded. This also	
	allows to bound the fist term on the right hand side with the square of the right hand side of \eqref{boundonehalfnew}. Let us now bound the second term: we have
	\begin{multline}\label{bound2}
	\sum_{\y\in \Nodes} \, \sum_{{\x \in \Nodes} \atop{\dist(\patch{\x},\patch{\y}) > 0}} \frac{2h^n}{\dist(\patch{\x},\patch{\y})^n} \| v^\x \|^2_{H^{1/2}_{00}(\patch{\x})} = \sum_{\x\in \Nodes}  \| v^\x \|^2_{H^{1/2}_{00}(\patch{\x})} \,  \sum_{{\y \in \Nodes} \atop{\dist(\patch{\x},\patch{\y}) > 0}} \frac{2h^n}{\dist(\patch{\x},\patch{\y})^n}  \\[2mm]
	 \lesssim \Big(\max_{\x \in \Nodes}\,   \sum_{{\y \in \Nodes} \atop{\dist(\patch{\x},\patch{\y}) > 0}} \frac{2h^n}{\dist(\patch{\x},\patch{\y})^n} \Big) \sum_{\x\in \Nodes} \| v^\x \|^2_{H^{1/2}_{00}(\patch{\x})}.
	\end{multline}
	We then need to bound 
	\[
	\max_{\x \in \Nodes} \sum_{{\y \in \Nodes} \atop{\dist(\patch{\x},\patch{\y}) > 0}} \frac{2h^n}{\dist(\patch{\x},\patch{\y})^n}.
	\]
	
We observe that, by the shape regularity of $\Th$, $\dist(\patch\x,\patch\y) > 0$ implies that	
	\(
	\dist(\patch\x,\patch\y) \simeq \dist(\patch\x,\patch\y) + h
	\). 
	 Then, for $s \in \patch\x$ with $\dist(\x,\y) > 0$ we have that
\(
 \dist(\patch\x,\patch\y) \gtrsim | s - \y | 
\),
 so that
	\[
\frac {h^{n}}{\dist(\patch\x,\patch\y)} \simeq h \int_{\patch{\x}} | s - \y |^{-n} \, ds,
	\]
	whence
	\[
\sum_{{\x \in \Nodes} \atop{\dist(\patch{\x},\patch{\y}) > 0}} \frac{2h^n}{\dist(\patch{\x},\patch{\y})^n} \lesssim h \int_{\Gamma\setminus \patch\y} | s - \y |^{-n} \, ds \lesssim 1.
\]
Bounding the second term on the right hand side of \eqref{bound1} by combining this bound with \eqref{bound2},  we obtain \eqref{boundonehalfnew}.
\end{proof}

\begin{remark}
Simply  switching, in  bound \eqref{nonoptimal}, from an $\ell^1$ to an $\ell^2$ norm by the standard norm equivalence results for finite dimensional spaces,  would only yield the weaker bound
	\[
	\vert v \vert_{1/2,\G} \leq  \#(\Nodes)^{1/2} \sqrt{\sum_{\y\in \Nodes} \| v^{\y} \|^2_{H^{1/2}_{00}(\patch{\y})}} \lesssim h^{(1-n)/2}  \sqrt{\sum_{\y\in \Nodes} \| v^{\y} \|^2_{H^{1/2}_{00}(\patch{\y})}} .
	\]
	\end{remark}

\

\NOTE{
\[1 \sim | \patch{\y} | \| \phistar{ \y} \|_{0,\infty} \sim h^{n-1}  \| \phistar{ \y} \|_{0,\infty}, \quad \longrightarrow \| \phistar{\y} \|_{0,\patch{\y}} \sim h^{(n-1/)2} \| \phistar {\y }\|_{0,\infty} \sim h^{1-n} h^{(n-1)/2} \sim h^{1/2 - n/2}
	\]
}

The following Theorem extends the localization results by Faermann \cite{faermann2000localization,faermann2002localization}.

\newcommand{\Phistar}{\Phi^\star}

\begin{theorem}\label{thm:pos} 
	Let $\Phistar = \{\phistar{\y}, \ \y \in \Nodes\} \subset L^2(\G)$ denote a set functions, not necessarily assumed to be linearly independent, satisfying 
	\begin{equation}\label{assumptionphistar}	\int_{\G} \phistar{\y}(s) \,ds= 1, \qquad \supp \phistar{\y} \subseteq \patch{\y} \quad \text{ and }\quad \| \phistar{\y} \|_{0,\patch{\y}} \simeq h^{1/2-n/2}.\end{equation}
	Then, for all $w \in H^{1/2}(\G)$ satisfying
	\[
	\langle w , \phistar{\y} \rangle_\Gamma = 0, \quad \text{ for all }\ \y \in \Nodes,
	\] it holds that
	\begin{equation}\label{boundonehalfextended}
		\| w \|_{1/2,\G}^2 \simeq  \sum_{\y \in \Nodes} | w  |_{1/2,\patch{\y}}^2.
	\end{equation}
\end{theorem}

\begin{proof}
	We know (see \cite{faermann2000localization,faermann2002localization}) that for all $w \in H^{1/2} (\G)$ it holds that
	\begin{equation}\label{hffaermann}
	\| w \|^2_{1/2,\G} \lesssim \sum_{\y \in \Nodes} | w |^2_{1/2,\patch{\y}} + \sum_{\y\in \Nodes} h^{-1} \| w \|^2_{0,\patch{\y}}.
	\end{equation}
	Now, letting $ \bar w = | \patch{\y} |^{-1} \int_{\patch{\y}} w(s)\, ds$ denote the average of $w$ on $\tau$, we have
	\begin{multline*}
		\| w \|_{0,\patch{\y}}^2 = \langle w\, ,\, w - 
		| \patch{\y} | \bar w \, 		\phistar{\y} \rangle_{\patch{\y} }  = \langle  w - \bar w,
		w -
		 | \patch{\y} | \bar w \, 		\phistar{\y} \rangle_{\patch{\y}} \\[2mm]
		\leq \| w  - \bar w \|_{0,\patch{y}} \| w -
	| \patch{\y} | \bar w \, 		\phistar{\y} \|_{0,\patch{\y}} 
		\lesssim h^{1/2} | w  |_{1/2,\patch{\y}} \Big (\| w \|_{0,\patch{\y}} + 	{|\patch{\y} |\, | \bar w |} \, \| \phistar{\y} \|_{0,\patch{y}} \Big)\\[2mm]
		\lesssim h^{1/2} | w  |_{1/2,\patch{\y}}   \| w \|_{0,\patch{\y}}  (1 + | \patch{\y} | ^{1/2}\,h^{1/2-n/2} ),
	\end{multline*}
	whence, as $| \patch\y | \simeq h^{n-1}$, 
	\[
	\| w \|_{0,\patch{\y}} \lesssim h^{1/2} | w  |_{1/2,\patch{\y}}.
	\]
Using this bound  in \eqref{hffaermann} gives us \eqref{boundonehalfextended},
	where the second inequality (i.e. the upper bound of the sum on the right hand side by a constant times the left hand side) immediately descends from the definition of the $H^{1/2}(\G)$ norm, thanks, once again, to the observation that, by the shape regularity of the mesh $\Th$, the number of patches containing a given triangle/tetrahedron is uniformly bounded.
\end{proof}

A particularly relevant application of the previous theorem is given by the following corollary.

\begin{corollary}\label{cor:pos} Let $\Pi : H^{1/2}(\G) \to H^{1/2}(\G)$ be a (possibly oblique) bounded projector. Assume that there exists a not necessarily linearly independent set \  $\Phistar = \{
	\phistar{\y}, \ \y \in \Nodes
	\} \subset L^2(\G)$ satisfying \eqref{assumptionphistar}, and such that $\Pi$ is orthogonal to $\Span \Phistar$, that is such that
	\[\langle v - \Pi(v)\, ,\, \phistar{\y} \rangle_\G = 0, \qquad \text{ for all }\y \in \Nodes.\]
	Then we have the following norm equivalence, valid for all $v \in H^{1/2}(\G)$:
	\begin{equation}\label{norm_equivalence}
	\| v \|_{1/2,\G}^2 \simeq \| \Pi(v) \|_{1/2,\G}^2 + \sum_{\y\in \Nodes} | v - \Pi(v)  |_{1/2,\patch{\y}}^2.
	\end{equation}
\end{corollary}

Remark that, under suitable shape regularity conditions on the mesh (see \cite{crouzeix1987stability,bramble2002stability,carstensen2002merging}), the $L^2$ projection on the space $L_h$ of continuous piecewise linears on $\Th$ satisfies the assumptions of Corollary \ref{cor:pos}. A similar result also holds for
 the oblique projection on  $L_h$ 
 taken orthogonally to the space of discontinuous piece wise constants on the polygonal Voronoi mesh dual to $\Th$ \cite{steinbach2002generalized}. Indeed it is not difficult to check that in the proof of Theorem \ref{thm:pos}, the triangular mesh $\Th$ in the definition of the set $\Phi$ can be replaced by its dual mesh.

%
%
%
%
%

\newcommand{\tphi}[1]{\widetilde\phi_{#1}}

\newcommand{\tPi}{\widetilde \Pi}
\renewcommand{\v}{\mathbf{v}}

\section{Localization of the $H^{-1/2}(\G)$ seminorm}

We aim at proving a localization result similar to Theorem \ref{thm:pos} for the dual space $H^{-1/2}(\G)$. We start be proving the following Proposition.

\begin{proposition}
	For all $\phi \in H^{-1/2}(\G)$
	it holds that
	\[
	\sum_{\y \in \Nodes} \| \phi \|^2_{\tH^{-1/2}(\patch\y)} \lesssim \| \phi  \|^2_{-1/2,\G}.
	\]
\end{proposition}

\begin{proof} Using Lemma \ref{lem:locposnew}, and denoting by $\v = (v_\y)_{\y\in \Nodes} \in \prod_{\y\in \Nodes} H^{1/2}_{00} (\patch{\y})$  the vector of functions whose entries are the $v_\y$s, we can write
	\begin{multline*}
	\left(\sum_{\y \in \Nodes} \| \phi \|^2_{\tH^{-1/2}(\patch\y)}\right)^{1/2} = 
		\left(\sum_{\y \in \Nodes}\, \left(\sup_{{v^\y \in H^{1/2}_{00}(\patch{\y})}\atop {v^ \y \not= 0}} \frac{\langle \phi,v^\y \rangle }{\| v^\y \|_{H^{1/2}_{00}(\patch{\y})}} \right)^2\right)^{1/2}\\[2mm] \lesssim\sup_{{\v \in \prod_{\y} H^{1/2}_{00}(\patch{\y})}\atop {\v \not = 0}} \frac{ \sum_{\y} \langle \phi,v^\y \rangle }{\sqrt {\sum_{\y} \| v^\y \|^2_{H^{1/2}_{00}(\patch{\y})}}} 
		=
		\sup_{{\v \in \prod_{\y} H^{1/2}_{00}(\patch{\y})}\atop {\v \not = 0}} \frac{  \langle \phi, \sum_{\y} v^\y \rangle }{\sqrt {\sum_{\y} \| v^\y \|^2_{H^{1/2}_{00}(\patch{\y})}}} \lesssim 
	\\[2mm]	\sup_{{\v \in \prod_{\y} H^{1/2}_{00}(\patch{\y})}\atop {\v \not = 0}} \frac{  \langle \phi, \sum_{\y} v^\y \rangle }{ \|{\sum_{\y} v^\y \|_{1/2,\G}}} \leq \sup_{{v \in H^{1/2}(\G)}\atop{v \not= 0}} \frac{  \langle \phi, v \rangle }{ \| v \|_{1/2,\G}} = \| \phi  \|_{-1/2,\G},
	\end{multline*}
which is the desired bound.
\end{proof}

We can now prove the following theorem.
\begin{theorem}\label{thm:negativenormlocalization} Let  $\Phi = \{\b{\y},\ \y \in \Nodes\} \subset H^{1/2}(\G)$ be a set of functions satisfying, for all $\y \in \Nodes$, $\b{\y} \in W^{1,\infty}_0(\patch{\y})$, with
	\[ 
	\| \b{\y} \|_{0,\infty,\patch\y} \lesssim 1, \quad | \phi^\y |_{1,\infty,\patch\y} \lesssim h^{-1}, \qquad \text{ and such that }\quad \sum_{\y \in \Nodes} \b{\y} = 1.\]  Then 
for all $\zeta \in H^{-1/2}(\G)$ satisfying	\[\langle \zeta , \b{\y} \rangle = 0, \qquad \text{ for all }\quad \y \in \Nodes,\]
it holds that
	\begin{equation}\label{boundnegativenorm}
	\| \zeta \|_{-1/2,\G}^2  \simeq \sum_{\y \in \Nodes} \| \zeta \|^2_{\tH^{-1/2}(\patch\y)}.
	\end{equation}
\end{theorem}

\

\newcommand{\p}{\phi^\y}
\newcommand{\w}{w^\y}

\newcommand{\mapy}{\Phi^{\y}}
\newcommand{\hp}{\widehat{\phi}}
\newcommand{\hv}{\widehat v}
\begin{proof}
Let $\zeta \in H^{-1/2}(\G)$ 
with, for all $\y \in \Nodes$, $\langle \zeta,\b{\y}\rangle = 0$, 
and, for $v \in H^{1/2}(\G)$, let $\bar v^\y =| \patch{\y} |^{-1} \int_{\patch{\y} }v(s)\,ds$ denote the average on $\patch{\y}$ of $v$. We have $\p(v^\y -\bar v^\y ) \in H^{1/2}_{00} (\patch{y})$ and we claim that 
\begin{equation}\label{claim}
\| \p(v^\y -\bar v^\y ) \|_{H^{1/2}_{00}(\patch{y})} \lesssim | v |_{1/2,\patch{y}}.
\end{equation}
Then can write
\begin{multline*}
\| \zeta  \|_{-1/2,\G} = \sup_{{v \in H^{1/2}(\G)}\atop{v \not= 0}} \frac{\langle \zeta  , v\rangle_\G}{\| v \|_{1/2,\G} } 
= \sup_{{v \in H^{1/2}(\G)}\atop{v \not= 0}}  \frac{\langle \zeta  , \sum_{\y} \phi^\y (v - \bar v^\y) \rangle_\G}{\| v \|_{1/2,\G} } \\[2mm]
\leq 
\sup_{{v \in H^{1/2}(\G)}\atop{v \not= 0}}  \frac{ \sum_{\y}  \| \zeta  \|_{\tH^{-1/2}(\patch\y)}
\| \phi^\y (v - \bar v^\y) \|_{H^{1/2}_{00}(\patch\y)}} 
{\| v \|_{1/2,\G}}\\
 \lesssim \sqrt{\sum_\y \| \zeta  \|^2_{\tH^{-1/2}(\patch\y)}}
\sup_{{v \in H^{1/2}(\G)}\atop{v \not= 0}} 
\frac{
\sqrt{\sum_\y \|  \phi^\y (v - \bar v^\y) \|^2_{H^{1/2}_{00}(\patch\y)}} 
}
{
\| v \|_{1/2,\G}
}\\ \lesssim \sqrt{\sum_\y \| \zeta  \|^2_{\tH^{-1/2}(\patch\y)}} \sup_{{v \in H^{1/2}(\G)}\atop{v \not= 0}} 
\frac{
\sqrt{\sum_\y | v |^2_{1/2,\patch\y} }
}
{
\| v \|_{1/2,\G}, 
} \lesssim \sqrt{\sum_\y \| \zeta  \|^2_{\tH^{-1/2}(\patch\y)}}.
\end{multline*}

We then need to show that\eqref{claim} holds.
Letting $\w  = v^\y - \bar v^\y$ and using  \eqref{onehalfzznorm} we have
\begin{multline*}
\| \p\w \|_{H^{1/2}_{00}(\patch\y)} \lesssim\underbrace{ \int_{\patch\y}
\,ds \int_{\patch\y}\,dt  \frac{| \p(s) \w(s) - \p(t) \w(t)|^2}{| s - t |^n} }_I	+ \underbrace{h^{n-2} \int_{\patch\y} \frac{|\p(s)\w(s)|^2}{\dist(s,\partial\patch\y)^{n-1}}\,ds}_{II}.
	\end{multline*}

%
%
To bound $I$ we observe that, adding and subtracting $\p (s) \w(t)$ we can write
\begin{multline*}
I \lesssim 
	\int_{\patch\y}
\,ds \int_{\patch\y}\,dt  \frac{| \p(s) |^2 |\w(s) - \w(t) |^2 }{| s - t |^n}
+
	\int_{\patch\y}
\,ds \int_{\patch\y}\,dt  \frac{  |\p(s) - \p(t) |^2 | \w(t)|^2}{| s - t |^n}
	\\[2mm] \lesssim
\| \p \|_{0,\infty,\patch\y}^2 | \w |^2_{1/2,\patch\y}
	+
\| \p \|^2_{1,\infty,\patch\y}	 \int_{\patch\y}\,dt  | \w(t)|^2 \int_{\patch\y}
\,ds\, {| s - t |^{2-n}}
	\\[2mm] \lesssim
 | \w |^2_{1/2,\patch\y} + h^{-1}  \| \w \|_{0,\patch\y}^2 \lesssim 
	 | \w |_{1/2,\patch\y}^2,
	\end{multline*}
where we used the bound $|\p(s) - \p(t) | \leq \| \p \|_{1,\infty,\patch\y} | s - t |$, the 
fact that, for $t \in \patch\y$ we have that $\int_{\patch\y} | s - t |^{2-n} \,ds \lesssim h$ (this holds for $n=2,3$), and the Poincar\'e inequality \eqref{poincareavfree}.

\NOTE{ For $n = 2$, $\int_{\patch\y} | s - t |^{2-n} \,ds = | \patch{\y} | \lesssim h$. For $n = 3$, letting for simplicity $t = (0,0)$ we have
\[
\int_{\patch\y} | s - t |^{-1} \,ds \leq \int_B(t,ch) | s - t |^{-1} \,ds  = \int_0^{2\pi} \,d\theta \int_{0}^{ch} \rho \,d\rho \,| \rho |^{-1} = 
2\pi \int_0^{ch} 1 \,d\rho \lesssim h. 
\]
}

To bound $II$, we let $\sigma \in \partial \patch\y$ be such that $\dist(s,\patch\y) = | s-\sigma |$, and we can write
\[
| \p(s) | = | \p(s)- \p(\sigma) | \leq \| \p \|_{1,\infty,\patch\y} | s - \sigma |  = \| \p \|_{1,\infty,\patch\y} d(s,\partial\patch\y),
\]
and then
\[\frac{| \p(s) |^2}{\dist(s,\partial\patch\y)^{n-1}} \leq \| \p \|^2_{1,\infty,\patch\y} \dist(s,\patch\y)^{3-n} \lesssim h^{-2} h^{3-n}.
\]
Then we have 
\begin{multline*}
	h^{n-2} \int_{\patch\y} \frac{|\p(s)\w(s)|^2}{\dist(s,\partial\patch\y)^{n-1}}\,ds \lesssim h^{n-2}\, \sup_{s \in \patch\y} \frac{| \p(s) |^2}{\dist(s,\partial\patch\y)^{n-1}}	\| \w \|_{0,\patch\y}^2\\[2mm] \lesssim h^{n-2} h^{1-n} \| \w \|_{0,\patch\y}^2 \lesssim | \w |_{1/2\patch\y}^2,
\end{multline*}
finally yielding
\[
\| \p \w \|_{H^{1/2}_{00}(\patch\y) }^2 \lesssim | \w |_{1/2,\patch\y}^2 = | v |_{1/2,\patch\y}^2,
\]
which gives us \eqref{claim} and, hence, finally, \eqref{boundnegativenorm}.
\end{proof}

Also here we can exploit the localization result given by Theorem \ref{thm:negativenormlocalization} to derive equivalent norms, this time  for $H^{-1/2}(\G)$. We have the following corollary.

\begin{corollary}\label{cor:equivnegative} Let $\tPi : H^{-1/2}(\G) \to H^{-1/2}(\G)$ be a bounded projector satisfying the following assumption: there exist functions  $\b{\y}\in W_0^{1,\infty}(\patch\y)$, $\y \in \Nodes$, with
	\[\sum_{\y \in \Nodes} \b{\y} = 1, \qquad\| \b{\y} \|_{0,\infty,\patch\y} \lesssim 1, \qquad | \phi^\y |_{1,\infty,\patch\y} \lesssim h^{-1},\]  such that 
	for all $\zeta \in H^{-1/2}(\G)$,	\[\langle \zeta - \tPi(\zeta), \b{\y} \rangle = 0.\]
	Then, for all $\zeta \in H^{-1/2}(\G)$
	\[
\| \zeta \|_{-1/2,\G}  \simeq \| \tPi \zeta \|_{-1/2,\G}^2 + \sum_{\y \in \Nodes} \| \zeta - \tPi \zeta \|^2_{\tH^{-1/2}(\patch\y)}.
	\]
\end{corollary}
Both the $L^2$ projection on the space of continuous piecewise linears and the oblique projection onto the space of piecewise constants on the dual mesh  orthogonally to the space of continuous piecewise linears satisfy the assumptions of Corollary \ref{cor:equivnegative}.

\bibliographystyle{plain}
\bibliography{biblio}

\end{document}